\documentclass[graybox]{svmult}

\usepackage{mathptmx}       
\usepackage{helvet}         
\usepackage{courier}        

\usepackage{makeidx}         
\usepackage{graphicx}        
\usepackage{multicol}        
\usepackage[bottom]{footmisc}

\usepackage{amsmath, amssymb, mathtools}
\usepackage{cite}
\usepackage{bm} 

\begin{document}

\title{Finite Volume Method for a System of Continuity Equations Driven by Nonlocal Interactions}
\titlerunning{Finite Volume for System of Nonlocal Continuity Equations}

\author{Anissa El Keurti and Thomas Rey}

\institute{
Anissa El Keurti and Thomas Rey
\at 
Univ. Lille, CNRS, UMR 8524, Inria – Laboratoire Paul Painlevé\\
  F-59000 Lille, France
\email{elkeurti.anissa@gmail.com, thomas.rey@univ-lille.fr}
}

\maketitle

\abstract{	
	We present a new finite volume method for computing numerical approximations of a system of nonlocal transport equation modeling interacting species.
	This method is based on the work [F. Delarue, F. Lagoutière, N. Vauchelet, \textit{Convergence analysis of upwind type schemes for the aggregation equation with pointy potential}, Ann. Henri. Lebesgue 2019], where the nonlocal continuity equations are treated as conservative transport equations with a nonlocal, nonlinear, rough velocity field. 
	We analyze some properties of the method, and illustrate the results with numerical simulations.
  \keywords{Upwind finite volume method, system of aggregation equations, population dynamics, continuity equations, measure-valued solutions.}
  \\[5pt]
  {\bf MSC }(2010){\bf:} 
    45K05, 
    65M08, 
    65L20, 
    92D25. 
  }

\section{A Nonlocal Predator-Prey Model}


We consider a system of nonlocal equations modeling the swarming dynamics of species which interact with each others through attractive/repulsive potentials (such as predators and preys). The system is an extension of the well-known aggregation equation \cite{bertozzi2009blow}, and can be written in the following form:
\begin{equation}
	  \left\{
				  \begin{aligned}
				    & \partial_t \rho_1 + \text{div}(\rho_1 (\nabla W_1*\rho_1 + \nabla K*\rho_2))= 0, && \rho_1(0, \cdot) = \rho_1^{in},
				     \\
				    & \partial_t \rho_2 + \text{div}(\rho_2 (\nabla W_2*\rho_2- \beta \nabla K*\rho_1))= 0, && \rho_2(0, \cdot) = \rho_2^{in},  
				  \end{aligned} \right.
\label{ppsys}				    
\end{equation}
where $\rho_1(t,x)$ and  $\rho_2(t,x)$ are probability measures that model the density of species $1$ and $2$ (respectively predators and preys), for $x \in \mathbb{R}^d$, $t \in \mathbb{R}$. 
This model was introduced in \cite{ex2species}, where it was derived from a system of $N$ interacting particles.
It has since been mathematically studied in \cite{2species, CarrilloDiFrancescoEspositoFagioliSchmidtchen}. 

The functions $W_\alpha, \ K:\mathbb R^d \to \mathbb R_+$, $\alpha \in \{1,2\}$ denote respectively the \emph{intra-specific} interaction potentials of the species $\alpha$, and the \emph{inter-specific} interaction potential. 
The intra-specific potential $W_\alpha$ can be of \textit{attractive} (namely radial with a nonnegative derivative) or \textit{repulsive} type (radial with a nonpositive derivative), depending on the gregarious behavior of species $\alpha$. The potential $K$ is of attractive type, modeling the fact that species $2$ flees species $1$ whereas species $1$ is attracted by species $2$. The parameter $\beta\in [0,1)$ expresses the mobility of species $1$. 

\section{Cauchy Theory}

\begin{definition}
  A function $W: \mathbb R^d \to \mathbb R$ is called a \emph{pointy} potential if it satisfies the following properties: 
\begin{enumerate}
\item $W$ is Lipschitz continuous, symmetric and $W(0)=0$;
\item $W$ is $\lambda$-convex for some $\lambda \le 0$ (namely $W-\frac{\lambda}{2}|\cdot|^2$ is convex);
\item $W \in \mathcal C^1(\mathbb{R}^d \setminus \{0\})$.
\end{enumerate}
\label{pointy}
\end{definition} 

Let us assume that $W_\alpha$, $\alpha \in \{1,2\}$, and $K$ are pointy potentials as in Def. \ref{pointy}. These potentials being Lipschitz, there exist $\omega_{\alpha,\infty}$ and $\kappa_{\infty}$ such that for all $x\ne 0$:
\begin{equation}
    |\nabla W_{\alpha}(x)| \le \omega_{\alpha,\infty}, \quad  |\nabla K(x)| \le \kappa_{\infty}.
\label{cond_bornitudePP}
\end{equation}
Let us also define the \emph{macroscopic velocities} $\widehat{a_{\rho_1}}$ and $\widehat{a_{\rho_2}}$ as
\begin{gather}
    \widehat{a_{\rho_1}}(t,x):= - \int_{\mathbb{R}^d} \left( \widehat{{\nabla W}}_{\alpha}(x-y)\ \rho_1(t,y) + \widehat{{\nabla K}}(x-y)\ \rho_2(t,y) \right) dy, \\
      \widehat{a_{\rho_2}}(t,x):= - \int_{\mathbb{R}^d} \left( \widehat{{\nabla W}}_{\alpha}(x-y)\ \rho_2(t,y) - \beta \widehat{{\nabla K}}(x-y)\ \rho_1(t,y) \right) dy,
    \label{a_rhoPP}
\end{gather}
where we denoted for a pointy potential $W$ the following extension:
  \[ \widehat{{\nabla W}}(x) = \left\{
    \begin{aligned}
          \nabla W_{\alpha}(x) & \text{ for } x \ne 0, \\
            0 & \text{ for } x = 0. \\
    \end{aligned} \right.
    \]

Existence theory for problem \eqref{ppsys} has been studied in \cite{ex2species} in the case of $\mathcal{C}^1$ pointy potentials. 
Uniqueness was obtained in \cite{vauch2species} by introducing duality solutions. This approach will  allow to prove the convergence of our numerical scheme \eqref{rho_numPP}. 
Using the theory of Filippov characteristics, one can also prove the following general result:
\begin{theorem}[From \cite{princ2}]
\label{thePP} 
Let $W_\alpha$, $\alpha \in \{1,2\}$, and $K$ be pointy potential that satisfy \eqref{cond_bornitudePP}, and $\rho^{in}_\alpha\in \mathcal{P}_2(\mathbb{R}^d)$. There exist unique probability measures $\rho_\alpha$ that are global distributional solutions to the following system of transport equations: 
\begin{equation}
	  \left\{
\begin{aligned}    
   {\partial t} \rho_1 + \text{div}(\widehat{a_{\rho_1}} \rho_1) =0, \quad \rho_1(0,\cdot)=\rho_1^{in}, \\ 
   {\partial t} \rho_2 + \text{div}(\widehat{a_{\rho_2}} \rho_2) =0, \quad \rho_2(0,\cdot)=\rho_2^{in}.
   \label{rho_solPP}
  \end{aligned} \right.  
\end{equation}
\end{theorem}

\section{Numerical Scheme}

 We shall now apply the numerical scheme introduced in \cite{delarue2017convergence} for approximating solutions to the classical (single species) aggregation equation to the system \eqref{ppsys}. Let us introduce a cartesian mesh $(C_J)_{J \in \mathbb Z^d}$ of $\mathbb{R}^d$, with step $\Delta x_i$ in the direction $i\in\{1,\ldots,d\}$, and $\Delta x = \max \Delta x_i$. 
 The center of a given cell $C_J$ will then be defined by $x_j := (J_1 \Delta x_1, \ldots, J_d \Delta x_d)$.
Let also $e_i := (0,\ldots,0,1,0,\ldots,0)$ be the $i$th vector of the canonical basis.
 
 For an initial probability measure $\rho^{in}_\alpha \in \mathcal P_2(\mathbb{R}^d)$, $\alpha \in \{1,2\}$, we define $\rho_{\alpha,J}^0$  as the cell average values of $\rho^{in}_\alpha$ over the cell $C_J$ : 
\begin{equation}
    \rho_{\alpha,J}^0 = \frac 1 {m\left (C_J\right)} \int_{C_J} \rho^{ini}_1(dx) \geq 0.
    \label{rho_iniPP}
\end{equation}
 Given an approximation $({\rho_\alpha}_J^n)_{J \in \mathbb{Z}^d}$ of the cell averages of $\rho_\alpha(t^n, \cdot)$ at a given time $t^n=n \Delta t$, we compute ${\rho_\alpha}_J^{n+1}$ as:
\begin{equation}
 \left\{
				  \begin{aligned}
    &{\rho_1}_{J}^{n+1}= {\rho_1}_J^n - \sum_{i=1}^d \frac{\Delta t}{\Delta x_i} \Big( && ({{a_1}_i^n}_J)^+{\rho_1}_J^n - ({{a_1}_i^n}_{J+e_i} )^-{\rho_1}_{J+e_i}^n \\ 
    & && -({{a_1}_i^n}_{J-e_i})^+{\rho_1}_{J-e_i}^n + ({{a_1}_i^n}_J)^-{\rho_1}_J^n \Big), \\
    &{\rho_2}_{J}^{n+1}= {\rho_2}_J^n - \sum_{i=1}^d \frac{\Delta t}{\Delta x_i} \Big( && ({{a_2}_i^n}_J)^+{\rho_2}_J^n -({{a_2}_i^n}_{J+e_i})^-{\rho_2}_{J+e_i}^n \\
    & && -({{a_2}_i^n}_{J-e_i})^+{\rho_2}_{J-e_i}^n + ({{a_2}_i^n}_J)^-{\rho_2}_J^n \Big).
     \end{aligned} \right.
     \label{rho_numPP}
\end{equation} \\
where the discrete macroscopic velocities are defined as
\begin{equation}
\left\{
				  \begin{array}{ll}
     {{a_1}_i^n}_J = - \sum_{L \in \mathbb{Z}^d} \left({\rho_1}_{L}^n D_i {W_1}_J^L + {\rho_2}_{L}^n D_i K_J^L \right),  \\
     {{a_2}_i^n}_J = - \sum_{L \in \mathbb{Z}^d} \left({\rho_2}_{L}^n D_i {W_{2}}_J^L - \beta {\rho_1}_{K}^n D_i K_J^L \right),  
     \end{array} \right.
      \label{a_numPP2}
\end{equation}
with $D_i W_J^K := \partial_{x_i} \widehat{W} (x_J - x_K)$ for a pointy potential $W$.

\begin{lemma}

If $W_\alpha$, $\alpha \in \{1,2\}$,  and $K$ are pointy potentials and the following CFL condition holds:
\begin{equation}
    (\omega_{\alpha,\infty}+ \kappa_{\infty}) \sum_{i=1}^d \frac{\Delta t}{\Delta x_i} \le 1,
    \label{CFLPP}
\end{equation}
one has the following properties for the scheme \eqref{rho_numPP}:
\begin{enumerate}
\item 
For $\rho^{in}_{\alpha} \in \mathcal{P}_2(\mathbb{R}^d)$ and $\rho_{\alpha,J}^0$ given by \eqref{rho_iniPP}, the sequences $(\rho_{\alpha,J}^n)_{J \in \mathbb{Z}^d, n \in \mathbb{N}}$ 
and $({{a_{\alpha}}_i^n}_J)_{J \in \mathbb{Z}^d, n \in \mathbb{N}, i=1,..,d}$ satisfy:
\[\rho_{\alpha,J}^n \ge 0, \quad |{{a_{\alpha}}_i^n}_J| \le (\omega_{\alpha,\infty} + \kappa_{\infty}), \quad i=1,\ldots,d,\]
and for all $n \in \mathbb{N}$, 
\[
  \sum_{J \in \mathbb{Z}^d} \rho_{\alpha,J}^n m(C_J) = \int_{\mathbb{R}} \rho^{in}_{\alpha}(dx).
\]

\item Conservation of the weighted center of mass:
\[\sum_{J \in \mathbb{Z}^d} x_J  (\beta \rho_{1,J}^n +  \rho_{2,J}^n) = \sum_{J \in \mathbb{Z}^d} x_J (\beta \rho_{1,J}^0 + \rho_{2,J}^0).\]

\end{enumerate}
\label{propPP}
\end{lemma}

\begin{proof}
  \begin{enumerate}
\item By summing the two equations of \eqref{rho_numPP} over all $J \in \mathbb{Z}^d$, one obtains the mass conservation.
  Then, writing both identities in \eqref{rho_numPP} as:
\begin{multline*}
{\rho_{\alpha}}_{J}^{n+1}= {\rho_{\alpha}}_J^n \big[  1- \sum_{i=1}^d |{{a_{\alpha}}_i^n}_J| \big] + \sum_{i=1}^d \frac{\Delta t}{\Delta x_i}  ({{a_{\alpha}}_i^n}_{J+e_i})^-{\rho_{\alpha}}_{J+e_i}^n \\
    + \sum_{i=1}^d \frac{\Delta t}{\Delta x_i} ({{a_{\alpha}}_i^n}_{J-e_i})^+{\rho_{\alpha}}_{J-e_i}^n , 
\end{multline*}
one proves by induction on $n$ that $\rho_{\alpha,J}^n \ge 0$ for all $J \in \mathbb{Z}^d$, $n \in \mathbb{N}$ under the CFL condition \eqref{CFLPP}. Indeed, by using the definition \eqref{a_numPP2}, one has
\[
|a_{\alpha_{i,J}}^n| \le (\omega_{\infty} + \kappa_{\infty}) \sum_{J \in \mathbb{Z}^d} \rho_{\alpha,J}^n = (\omega_{\infty} + \kappa_{\infty}) \sum_{J \in \mathbb{Z}^d} \rho_{\alpha,J}^0, \quad i \in \{1,..,d \},
\]
which concludes the proof by a convexity argument. 

\item Using a discrete integration by parts and \eqref{rho_numPP}, one has: 
\begin{multline*}
 \sum_{J \in \mathbb{Z}^d} x_J \rho_{\alpha,J}^{n+1} =  \sum_{J \in \mathbb{Z}^d} x_J \rho_{\alpha,J}^n - \sum_{i=1}^d  \frac{\Delta t}{\Delta x_i}  \sum_{J \in \mathbb{Z}^d} \left( ({{a_{\alpha}}_i^n}_{J})^+ {\rho_{\alpha}}_{J}^n (x_J - x_{J+e_i}) \right. \\ - \left.  ({{a_{\alpha}}_i^n}_{J})^-{\rho_{\alpha}}_{J}^n (x_{J-e_i} - x_J)  \right).
\end{multline*}
Since $x_J$ denote the cell centers, one has
\begin{multline}
 \sum_{J \in \mathbb{Z}^d} x_J \left( \beta \rho_{1,J}^{n+1}+ \rho_{2,J}^{n+1} \right) =  \sum_{J \in \mathbb{Z}^d} x_J \left(\beta \rho_{1,J}^n + \rho_{2,J}^n \right) \\ + {\Delta t}\sum_{i=1}^d   \sum_{J \in \mathbb{Z}^d}  \left({\beta {a_{1}}_i^n}_{J} {\rho_{1}}_{J}^n + {{a_{2}}_i^n}_{J} {\rho_{2}}_{J}^n \right).\label{eq:centerofmass}
\end{multline}
Summing over all the cells in \eqref{a_numPP2}, 
%
%
%
and since $\nabla W_\alpha$ and $\nabla K$ are odd, one obtains after reindexing:
\begin{align*}
\sum_{J \in \mathbb{Z}^d}  \beta {{a_{1}}_i^n}_{J} {\rho_{1}}_{J}^n + {{a_{2}}_i^n}_{J} {\rho_{2}}_{J}^n
&= \sum_{J \in \mathbb{Z}^d}  \sum_{L \in \mathbb{Z}^d} \left( \beta {\rho_{1}}_{J}^n {\rho_1}_{L}^n D_i {W_1}_J^L   +  {\rho_{2}}_{J}^n {\rho_2}_{L}^n D_i {W_2}_J^L   \right) \\ 
&= - \sum_{J \in \mathbb{Z}^d}  \sum_{L \in \mathbb{Z}^d} \left( \beta {\rho_{1}}_{J}^n {\rho_1}_{L}^n D_i {W_1}_L^J   +  {\rho_{2}}_{J}^n {\rho_2}_{L}^n D_i {W_2}_L^J  \right) \\
&= 0
\end{align*}
which yields the conclusion when plugged into \eqref{eq:centerofmass}.

\end{enumerate}

\end{proof}

We are now ready to prove the convergence of the scheme \eqref{rho_numPP}.

\begin{theorem} \label{the_numPP} 
 Let us assume that $W_\alpha$, $\alpha \in \{1,2\}$ and $K$ are pointy potentials, 
 and that the following CFL condition holds on the mesh $(C_J)$: 
\begin{equation*}
(\omega_{\alpha,\infty}+ \kappa_{\infty}) \sum_{i=1}^d \frac{\Delta t}{\Delta x_i} \le 1. 
\end{equation*}
Let $\rho_\alpha^{in} \in \mathcal{P}_2(\mathbb{R}^d)$ and $\rho_{\alpha,J}^0$ given by \eqref{rho_iniPP} for all $J \in \mathbb{Z}^d$ and define the empirical distribution as
\[ \rho_{\alpha, \Delta x}^n = \sum_{J \in \mathbb{Z}^d} \rho_{\alpha,J}^n \delta_{x_J}, \quad n \in \mathbb{N},\]
where  $((\rho_{\alpha,J}^n)_{J\in \mathbb{Z}^d})_{n \in \mathbb{N}}$ is given by \eqref{rho_numPP}.\

Then $\rho_{1, \Delta x}$ and $\rho_{2, \Delta x}$ converge weakly in  $\mathcal{M}_b([0,T]\times \mathbb{R}^d)$ towards respectively $\rho_1$ and $\rho_2$ which are the solutions to  \eqref{rho_solPP} as $\Delta x $ goes to 0.

\end{theorem}

\begin{proof}
Let us give the ideas behind this convergence proof, in the unidimensional case (inspired from \cite{vauch2species}). 

\begin{enumerate}
\item Extraction of a convergent subsequence.

The total variation of $\rho_{\alpha, \Delta x}$ is bounded and we can thus extract a subsequence of $\rho_{\alpha, \Delta x}$ that converges weakly towards $\rho_{\alpha} \in \mathcal{M}_b([0,T]\times \mathbb{R})$. 

\item Modified equations and Taylor expansion. 

We write the modified equation satisfied by $\rho_{\alpha, \Delta x}$ in terms of distributions. Let us consider $\phi \in C^{\infty}_c([0,T] \times \mathbb{R})$. By using the dual product in sense of distribution $<\cdot,\cdot>$, one has 
\[
<\partial_t\rho_{\alpha, \Delta x}, \phi> + <{\widehat{a}_{\alpha, \Delta x}} \rho_{\alpha, \Delta x}, \partial_x \phi> = 0,
\]
where $ \widehat{a}_{\alpha, \Delta x}= \sum_{n=0}^{N_T} \sum_{J \in \mathbb{Z}} \widehat{a}_{\alpha, J}^n \mathbf{1}_{[ t^n, t^{n+1}[}(t) \delta_{x_J}(x)$.
Taylor expanding $\phi$ allows to rewrite this equation in terms of distributions. One then bounds the different terms by using a straightforward adaptation of \cite[Lemma 6.2]{vauch2species} to this model. 

\item Passing to the limit. 

We finaly use \cite[Lemma 3.2]{vauch2species} to pass to the limit. The limit $\rho_{\alpha}$ thus satisfies \eqref{rho_solPP}. By uniqueness from Theorem \ref{thePP}, $\rho_{\alpha}$ is the unique solution of  \eqref{ppsys}. 
\end{enumerate}

\end{proof}

\section{Numerical simulations in 2D}

We implemented the scheme in 2 dimensions for a square grid and potentials such as the Newtonian potential $\mathcal N(x)=|x|$ (pointy and $0$--convex) , or  $W=1-e^{-|x|}$ (pointy and $-1$--convex). The grid in all the simulations is composed of $50 \times 50$ points, with $\Delta t=0.005$ (according to the CFL condition \eqref{CFLPP}).

%
%
%
%
%
%
%
%
%

\paragraph{\textbf{Test 1.} Evading preys.}
In Figure \ref{prey2}, we present simulations made with a Dirac delta as initial data to model a single predator, and a uniform distribution for preys:
\begin{equation}
  \label{eq:Initcond}
  \rho_1^{in} = \delta_0(x), \qquad \rho_1^{in} = \bm 1_{\mathcal{B}(0.2, 0.1)}. 
\end{equation}
We use Newtonian potentials $W_1=W_2=0.1\mathcal N$, $K= \mathcal N$ for inter and intra-specific interactions, with a mobility $\beta =0.3$. 
At the beginning of the simulation, we observe that the predator is getting closer to the preys. When the group of preys is close, the preys create a circular pattern around the the predators in order to run away from him. 

\begin{figure}
\begin{center}
\begin{tabular}{ll}
    \includegraphics[scale=0.37]{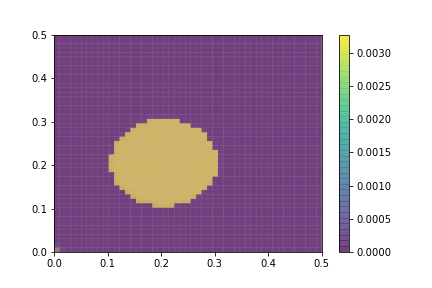}   &    \includegraphics[scale=0.37]{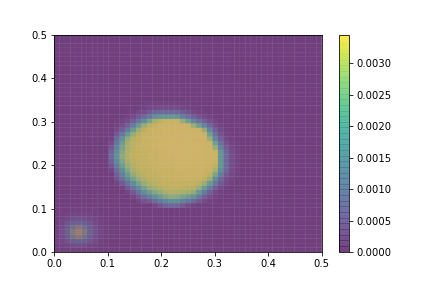}\\
  
    \includegraphics[scale=0.37]{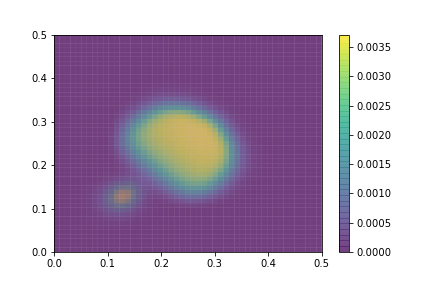}  &    

    \includegraphics[scale=0.37]{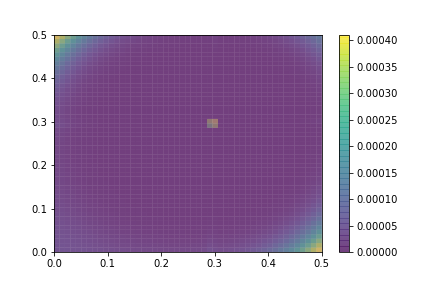}
\end{tabular}

\caption{\textbf{Test 1.} Newtonian potentials $W_1(x)=W_2(x)= 0.1|x|$,$K(x)=|x|$, $\beta=0.3.$ with a single predator at the origin, and an uniform distribution of preys as initial data.}
\label{prey2}
\end{center}
\end{figure}

\paragraph{Test 2. A more realistic potential for inter-specific interaction.}

In \cite{2species}, the authors introduced a potential $K$ that is more relevant in terms of modeling: 
\begin{equation}
    K(x) = 1-(|x|+1)e^{-|x|}.
\end{equation}
When the predator is far from the preys, the inter-specific interaction depends very weakly on the distance between preys and predator, and is almost constant. When the predator becomes closer to the preys, they become paralyzed, the potential being the close to $0$.  We performed simulations with an initial data given by \eqref{eq:Initcond} in Figure \ref{prey3}. We observe a similar behavior than in Figure \ref{prey2} in short time, but a convergence toward a single Dirac delta (The predator has gathered all the prey together) in large time. 

\begin{figure}
\begin{tabular}{ll}
    \includegraphics[scale=0.4]{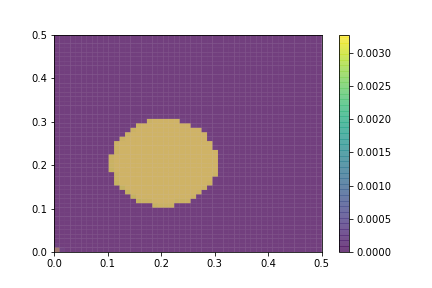}   &    \includegraphics[scale=0.4]{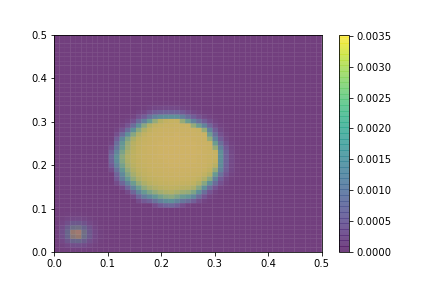}\\
  
    \includegraphics[scale=0.4]{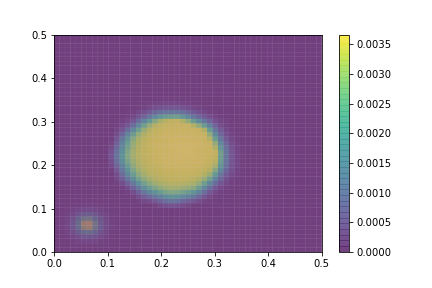}  &    \includegraphics[scale=0.4]{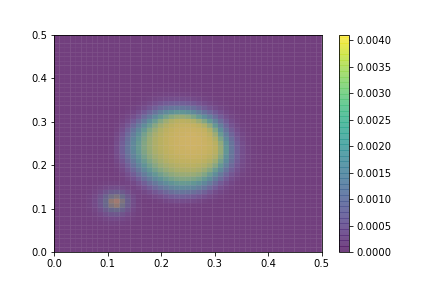}\\
   
    \includegraphics[scale=0.4]{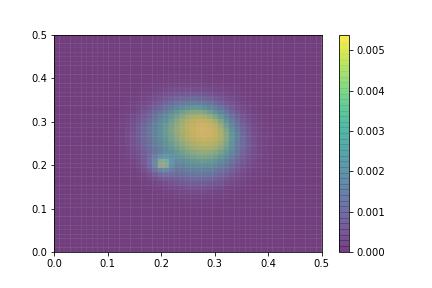}   &    \includegraphics[scale=0.4]{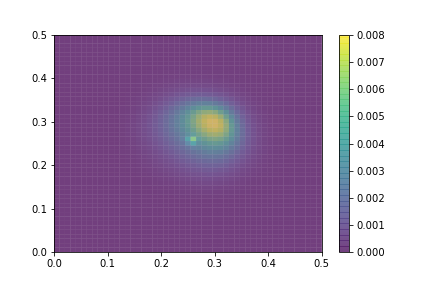}\\
    
    \includegraphics[scale=0.4]{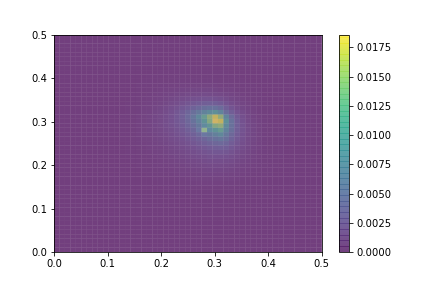}   &    \includegraphics[scale=0.4]{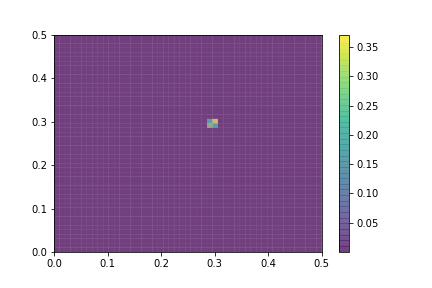}\\
     & 
\end{tabular}
\centering

\caption{\textbf{Test 2.} Newtonian potentials $W_1(x)=W_2(x)= 0.1 |x|$, ``fly-and-regroup'' potential $K(x)=1-(|x|+1)e^{-|x|} $, $\beta=0.3.$ with a single predator at the origin, and an uniform distribution of preys as initial data.}
\label{prey3}
\end{figure}

\begin{acknowledgement}
TR was partially funded by Labex CEMPI (ANR-11-LABX-0007-01) and ANR Project MoHyCon (ANR-17-CE40-0027-01).
\end{acknowledgement}

  \bibliographystyle{acm}
  \bibliography{Bib}
\end{document}